\documentclass[12pt,a4paper]{amsart}
\usepackage{amsmath,amsthm,amsfonts,mathtools,tikz}
\usepackage{hyperref,footmisc} 

\theoremstyle{theorem}
\newtheorem{theorem}{Theorem}

\newtheorem*{lemma*}{Lemma}
\newtheorem{corollary}[theorem]{Corollary}

\theoremstyle{definition}

\newcommand\fm{\mathfrak{m}}
\newcommand\fl{\mathfrak{l}}

\newcommand{\pl}[1]{\textbf{\small{plane}}\,{#1}}

\DeclareMathOperator{\bis}{bis}

\title{On Euler's rotation theorem}
\author[P.~Gothen and A.~Guedes~de~Oliveira]{\vspace{-0.5cm}Peter Gothen and Ant\'onio Guedes de Oliveira}
\address{CMUP and Department of Mathematics, Faculty of Sciences,
  University of Porto, 4169-007 Porto, Portugal}
\email{$\{\text{pbgothen,agoliv}\}$@fc.up.pt}

\date{07/09/2021}

\begin{document}

\maketitle

\begin{abstract}
It is well known that a rigid motion of the Euclidean plane can be
written as the composition of at most three reflections.
It is perhaps not so widely known that a rigid motion of $n$-dimensional Euclidean space can be written as the composition of at most $n + 1$ reflections.

The purpose of the present article is, firstly, to present a natural proof of this result in dimension $3$ by explicitly
constructing a suitable sequence of reflections, and, secondly, to show how a careful analysis of this construction provides a quick and pleasant geometric
path to Euler's rotation theorem, and to the complete
classification of rigid motions of space, whether orientation preserving or not.
We believe that our presentation will highlight the elementary nature of the results and hope that readers, perhaps especially those more familiar with the usual linear algebra approach, will appreciate the simplicity and geometric flavour of the arguments.
\end{abstract}

\section{On Euler's rotation theorem}

\edef\myindent{\the\parindent}
\edef\myparskip{\the\parskip}

\subsection{Decomposition of $3$-isometries}

It is well known that a rigid motion of the Euclidean plane can be
written as the composition of at most three reflections.  It is
perhaps not so widely known that a rigid motion of
  $n$-dimensional Euclidean space can be written as the composition of
  at most $n+1$ reflections.

The purpose of the present article is, firstly, to present a natural
proof of this result in dimension $3$ by explicitly
constructing a suitable sequence of reflections~\footnote{We
    opted to state and prove this theorem in three dimensions, but
    note that this proof can be easily adapted to fit any number of
    dimensions.} and, secondly, to show how a careful
analysis of this construction provides a quick and pleasant geometric
path to Euler's rotation theorem, and to the complete
classification of rigid motions of space, whether orientation preserving or not.

We believe that our presentation will highlight the elementary nature
of the results and hope that readers, perhaps especially those more
familiar with the usual linear algebra approach, will appreciate the
simplicity and geometric flavour of the arguments.

In view of the topic of the article any list of references is bound to
be inadequate, so we provide just two: our article \cite{GO} which
deals with the case of the plane, and the article \cite{PPR} which
gives a thorough discussion of Euler's theorem and, among
several proofs, includes Euler's original one and a modern one using
linear algebra.

Let $\pi=\pl{ABC}$ be the plane through three noncollinear points,
$A$, $B$ and $C$. We write $\sigma^{ABC}$ or $\sigma^\pi$ for the
reflection in $\pi$. If $A\neq B$ we write
$\bis \,\overline{\!AB}$ for the plane through the midpoint of
$\overline{AB}$ perpendicular to this line, which is formed by the
points at equal distance to $A$ and $B$. In particular, if
$A\neq B$, then $\sigma^{\bis\overline{\!AB}}(A)=B$ and vice-versa.

\begin{theorem}\label{TFI} Given points $A$, $A'$, $B$, $B'$, $C$ and $C'$ in space
  such that $A$, $B$ and $C$ are noncollinear and $|AB|=|A'B'|$,
  $|AC|=|A'C'|$ and $|BC|=|B'C'|$, there exist exactly two rigid
  motions sending $A$ to $A'$, $B$ to $B'$ and $C$ to $C'$. The first
  of these can be written as the composition of three reflections. The
  second one is obtained composing the first one with the reflection
  in the plane $A'B'C'$.
\end{theorem}
\begin{proof}
We start by noting that the location of any point $P$ is determined
  by its distances to any four given non-coplanar points: indeed if
  $Q\neq P$ had the same distances to the four given points, then
  these would belong to $\bis \,\overline{\!PQ}$ and thus be
  coplanar. It follows that a motion is determined by its action on
  any four non-coplanar points and, therefore, there exists no third
  motion with the stated properties.

Our proof now proceeds by first constructing a rigid motion $i$, written
as the composition of three reflections and satisfying $i(A)=A'$,
$i(B)=B'$ and $i(C)=C'$. The second motion will then be $j:=\sigma^{A'\!B'\!C'}\circ i \neq i$.

In the generic situation, when $A\neq A'$, we define $\alpha=\bis
\,\overline{\!AA'}$, so that $A'=\sigma^\alpha(A)$.

Next, supposing that $B^*:=\sigma^\alpha(B)$ is distinct from $B'$, we
define $\beta=\bis \,\overline{\!B^* B'}$, so that $B'=\sigma^\beta\circ\sigma^\alpha(B)$. The key point is now to
note that $\sigma^\beta(A')=A'$, because
$$|A'B^*  |=|\sigma^\alpha(A)\;\sigma^\alpha(B)|=|AB|=|A'B'|\,.$$
Hence $\sigma^\beta\circ \sigma^\alpha$ sends $A$ to $A'$ and $B$ to
$B'$. 

The final step is completely analogous. We define
$C^*=\sigma^\beta\circ\sigma^\alpha(C)$,
$\gamma=\bis\overline{C^*C'}$ (here supposing
$C^*\neq C'$). Again, $\sigma^\gamma(A')=A'$ and
$\sigma^\gamma(B')=B'$, because
\begin{align*}
&|A'C^*  |=|\sigma^\beta\!\circ\!\sigma^\alpha(A)\ \sigma^\beta\!\circ\!\sigma^\alpha(C)|=|AC|=|A'C'|\\
\intertext{ and similarly }
&|B'C^*  |=|\sigma^\beta\!\circ\!\sigma^\alpha(B)\ \sigma^\beta\!\circ\!\sigma^\alpha(C)|=|BC|=|B'C'|\,.
\end{align*}
Hence, setting $i=\sigma^\gamma\circ\sigma^\beta\circ\sigma^\alpha$,
we have $i(A)=A'$, $i(B)=B'$ and $i(C)=C'$.

Now, if $A=A'$ then we may instead define $\alpha=\pl{A'\!BC}$, if
$B^*=B'$ then we define $\beta=\sigma^\alpha(\pl{ABC})=\pl{A'\!B'\!\sigma^\alpha(C)}$, and
if $C^*=C'$ then we define $\gamma=\sigma^\beta\circ\sigma^\alpha(\pl ABC)=\pl{A'\!B'C'}$.  In all these
cases~\footnote{In each of the cases of coincidence, if we were to
  instead simply omit the respective reflection from our sequence, we
  would still obtain a motion with the desired properties. The
  reason why we do not do so, is that having the specific
  sequence of reflections will be essential in our proof of
  Theorem~\ref{theor:euler}.}, one still has $i(A)=A'$, $i(B)=B'$ and
$i(C)=C'$. Note that if $A=A'$, $B=B'$ and $C=C'$ then
$i=\sigma^{ABC}$.
\end{proof}

\subsection{Euler's rotation theorem}
Recall that a rotation about a line $\ell$
  is a  composition $\rho=\sigma^{\pi_2}\circ\sigma^{\pi_1}$ of
reflections, where the planes $\pi_1$ and $\pi_2$ intersect along the line
$\ell$, forming an angle which is half that of the rotation.

\begin{theorem}
\label{theor:euler}
Let $\fm$ be a rigid motion in space with a fixed point, $C$, and suppose $\fm$ is not the identity.
\begin{itemize}
\item{\emph{(Euler)}} If $\fm$ is an orientation preserving isometry then $\fm$ is a rotation about a line through $C$.
\item If $\fm$ does not preserve orientations then $\fm$ is either an
  inversion in $C$, a reflection in a plane through $C$, or a
  \emph{rotary reflection}~\footnote{Note that an inversion is
    the special case of a rotary reflection corresponding to the
    composition of reflections in 3 perpendicular planes.}, a
  reflection in a plane $\pi$ through $C$ followed (or
    preceded) by a rotation about a line through $C$ perpendicular to
  $\pi$.
\end{itemize}
\end{theorem}

\begin{proof}
Let $A$ be such that $B=\fm(A)\neq A$. 
Since $|AC|=|BC|$, if  for every point $A$ the  points $A$, $B$ and $C$
are collinear then for every point $A$, $C$ is the midpoint of
$\overline{AB}$. Hence, $\fm$ is the inversion in $C$.
Thus we may assume that $A$, $B$ and $C$ are non-collinear.
\\[-7.5pt]

\noindent
\begin{minipage}{.6\textwidth}
\setlength{\parindent}{\myindent}
\indent
Let us construct $i$ and $j$ as defined in the proof of
Theorem~\ref{TFI}, using the points $A$, $B$ and $C$, and their
images $A'=B$, $B'$ and $C'=C$ under $\fm$. Since $C$ is fixed by $\fm$, we
have $|AC|=|BC|=|B'C|$.
Since $\alpha=\bis\overline{AB}$ we have
$\sigma^\alpha(C)=C$ and, moreover, $B^*=\sigma^\alpha(B) = A$.
We can now see that
$C\in\beta$: in the case $A\neq B'$ because $\beta=\bis\overline{AB'}$
and $|AC| =|B'C|$, and in the case $B'=A$ because $\beta=\pl A'B'C$. It follows that
$C^*=\sigma^\beta\circ\sigma^\alpha(C)=C$. 
\end{minipage}\hspace{.2cm}
\begin{minipage}{.35\textwidth}
\centering
\begin{tikzpicture}[scale=.225]
\draw[fill] (0,0) node[red] {\includegraphics[width=4.5cm]{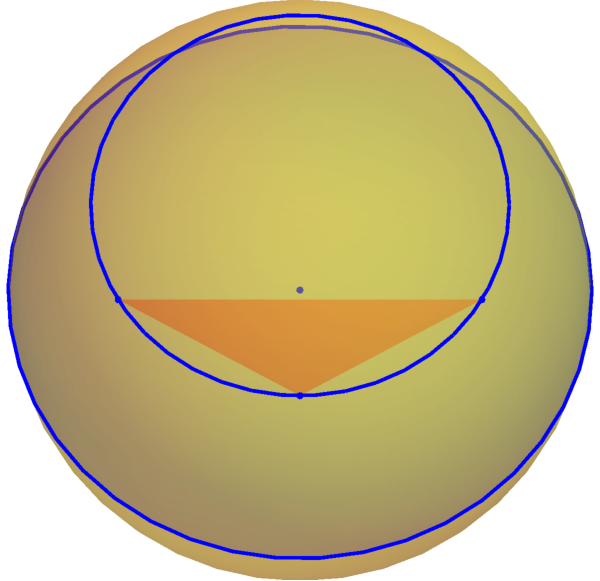}} circle (0pt) ;
\draw[fill] (0,0) node[above right=0pt] {\footnotesize$C$} circle (3pt) ;
\draw[fill] (6.05,-.3) node[right=0pt] {\footnotesize$A$} circle (3pt) ;
\draw[fill] (-0,-3.5) node[below=-2pt] {\footnotesize$B$} circle (3pt) ;
\draw[fill] (-6.05,-.3) node[left=0pt] {\footnotesize$B'$} circle (3pt) ;
\end{tikzpicture}
\end{minipage}
\smallskip

Thus $\gamma=\pl A'B'C'$,
and we conclude that
$j=\sigma^{A'B'C'}\circ\sigma^\gamma\circ\sigma^\beta\circ\sigma^\alpha=\sigma^\beta\circ\sigma^\alpha$, which is a rotation around the
line $\alpha\cap\beta$.

\noindent
\begin{minipage}{.6\textwidth}
\setlength{\parindent}{\myindent}
\indent
It remains to identify $i$.
If $B'=A$, then $\beta=\gamma$ and $i=\sigma^{\bis AB}$ is a
  reflection, so let us assume $B'\neq A$.
Let $D$ be the
midpoint of $\overline{AB}$, let $D'=\fm(D)$ be the midpoint of
$\overline{BB'}$ and let $\delta=\pl CDD'$. To see that $i$ is a
rotary reflection, we shall show that $\fm':=\sigma^\delta\circ i$ is a
rotation around the line $\ell$ through $C$ perpendicular to $\delta$.

Let $\tilde{B}=\sigma^\delta(B)=\fm'(A)$. If $\tilde{B}=B$, we have
$\delta=\gamma$ and we have the desired conclusion, noting that
$i=\sigma^{\gamma}\circ j$.
\end{minipage}\hfill
\begin{minipage}{.35\textwidth}
\centering
\begin{tikzpicture}[scale=.225]
\draw[fill] (0,0) node[red] {\includegraphics[width=4.5cm]{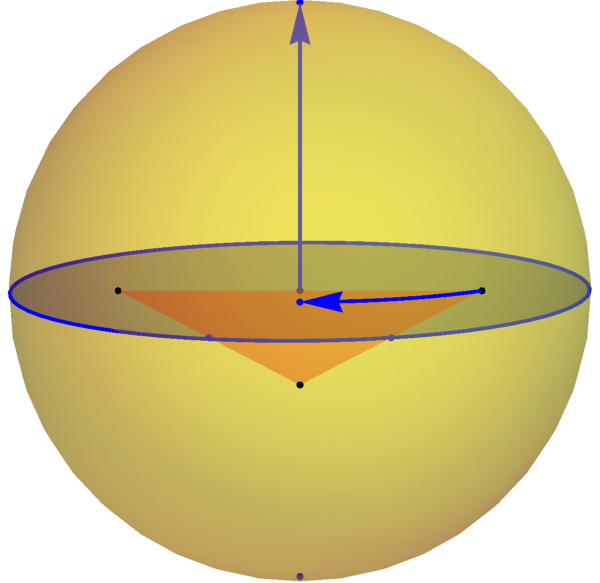}} circle (0pt) ;
\draw[fill] (0,0) node[above right=0pt] {\footnotesize$C$} circle (3pt) ;
\draw[fill] (6.05,0) node[right=0pt] {\footnotesize$A$} circle (3pt) ;
\draw[fill] (-0,-3.2) node[below=-2pt] {\footnotesize$B$} circle (3pt) ;
\draw[fill] (-6.05,0) node[left=0pt] {\footnotesize$B'$} circle (3pt) ;
\draw[fill] (-.15,-0.4) node[left=-2.5pt] {\footnotesize$\tilde{B}$} circle (3pt) ;
\draw[fill] (2.95,-1.55) node[below=-2.5pt] {\footnotesize$D$} circle (3pt) ;
\draw[fill] (-3.1,-1.4) node[below=-2.5pt] {\footnotesize$D'$} circle (3pt) ;
\end{tikzpicture}
\end{minipage}

Assume from now on
that $\tilde{B}\neq B$. Arguing as in the first part of the
proof for $\fm'$ (using that $\fm'(A) = \tilde{B}$, $\fm'(\tilde{B})=B'$ and
$\fm'(C)=C$) we see that there are two possibilities. The first one is
that
\begin{displaymath}
  \fm' = \sigma^{\bis\overline{AB'}}\circ\sigma^{\bis\overline{A\tilde{B}}}.
\end{displaymath}
Now observe that the planes $\bis\overline{AB'}$ and
$\bis\overline{A\tilde{B}}$ are perpendicular to $\pl A\tilde{B}B'$,
and therefore also perpendicular to the parallel plane $\delta$.
Thus $\rho:=\sigma^{\bis\overline{AB'}}\circ\sigma^{\bis\overline{A\tilde{B}}}$
is a rotation around $\ell$, and again we have the desired
conclusion. The second possibility is that 
\begin{displaymath}
  \fm' = \sigma^{C\tilde{B} B'}\circ \rho,
\end{displaymath}
which would imply that
\begin{math}
  \fm'(D) = \sigma^{C\tilde{B}B'}(D').
\end{math}
But, from
  $\fm'(D) = \sigma^\delta\circ\sigma^{C\tilde{B}B'}\circ\sigma^{\bis\overline{AB'}}\circ\sigma^{\bis\overline{A\tilde{B}}}(D)$, it is easy to check that
\begin{math}
  \fm'(D)=D'.
\end{math}
Thus we would have
$D'\in\pl C\tilde{B}B' \iff \tilde{B}\in\pl CD'B'=\pl CBB'$,
implying $B=\tilde{B}$, contrary to hypothesis.
\end{proof}
\subsection{Classification of $3$-isometries}

In the next corollary we complete the classification of rigid
  motions in space and for this we remind the reader that, given two
  points $A$ and $B$, the translation $\tau^{AB}$ which sends $A$ to
  $B$ can be constructed as follows: let $\sigma'$ be reflection in
  the plane through $B$ parallel to $\bis\overline{AB}$. Then
  $\tau^{AB}=\sigma'\circ\sigma^{\bis\overline{AB}}$. Recall that for
  points $A$, $B$ and $C$, one has
  $\tau^{BC}\circ\tau^{AB}=\tau^{AC}$.
\begin{lemma*}
Let $\rho$ be a non-trivial rotation about a line $\ell$ and let $\tau$ be a
translation in a direction perpendicular to $\ell$. Then the
composition $\tau\circ\rho$ is a rotation about a line parallel to $\ell$.
\end{lemma*}
\begin{proof}
  Write $\tau=\sigma_1\circ\sigma_2$ as the composition of reflections
  in parallel planes $\pi_1$ and $\pi_2$, and
  $\rho=\sigma_3\circ\sigma_4$ as the composition of reflections in
  planes $\pi_3$ and $\pi_4$ which intersect along $\ell$. Note that
  we can take for $\pi_3$ any plane through $\ell$, by changing
  $\pi_4$ accordingly.  By hypothesis $\ell$ is parallel to $\pi_1$
  and $\pi_2$ and thus we may take $\pi_3$ to be parallel to those
  two planes as well. It follows that
  $\sigma_3\circ\sigma_2\circ\sigma_1$ is a reflection $\sigma'$ in
  plane $\pi'$ parallel to $\sigma_3$ and hence
  $\tau\circ\rho=\sigma_4\circ\sigma'$ is a rotation.
\end{proof}

\begin{corollary}
\label{theor:chasles}
Let $\fm$ be a rigid motion in space different from the identity.
\begin{itemize}
\item{\emph{(Mozzi-Chasles)}} If $\fm$ is an orientation preserving
  isometry then $\fm$ is a screw displacement, that is, a
  rotation about a line $\ell$ followed (or preceded) by a (possibly
  trivial) translation in the direction of $\ell$.

\item If $\fm$ is not orientation preserving then $\fm$ is either an
  inversion, or a reflection in a plane, or a rotary reflection, or a
  \emph{glide plane operation}, that is, a reflection in a
  plane $\pi$ followed (or preceded) by a translation in $\pi$.
\end{itemize}
\end{corollary}
\begin{proof}
  In view of Theorem~\ref{theor:euler}, we may assume that $\fm$ has
  no fixed point.  Let us take any point $A$ and define a motion $\fl$ by
    $\fl(X)=\tau^{A'A}\circ\fm(X)$. Note that
    $\fm=\tau^{AA'}\circ\fl$.  Then $\fl(A)=A$ and, by
  Theorem~\ref{theor:euler}, $\fl$ fixes a plane~\footnote{The
    plane $\pi$ is not unique, in general.} $\pi$ which contains $A$.
  Let $A''\in\pi$ be the foot of the perpendicular to $\pi$
    through $A'$.  We now consider each of the possibilities for $\fl$
    given in Theorem~\ref{theor:euler}.
\begin{enumerate}
\item$\ $\\[-34.25pt]
\begin{minipage}{.5\textwidth}
  If $\fl$ is a rotation about a line through $A$, we may take $\pi$
  to be the perpendicular through $A$ to this line, and then, by the
  Lemma, $\tau^{AA''}\circ\fl$ is a rotation about a parallel line.
Hence, $\fm=\tau^{A''A'}\circ\tau^{AA''}\circ\fl$ is a screw
displacement.
\end{minipage}
\begin{minipage}{.425\textwidth}
\centering
\begin{tikzpicture}[scale=.225]
\draw[fill] (0,0) node[red] {\includegraphics[width=4.5cm]{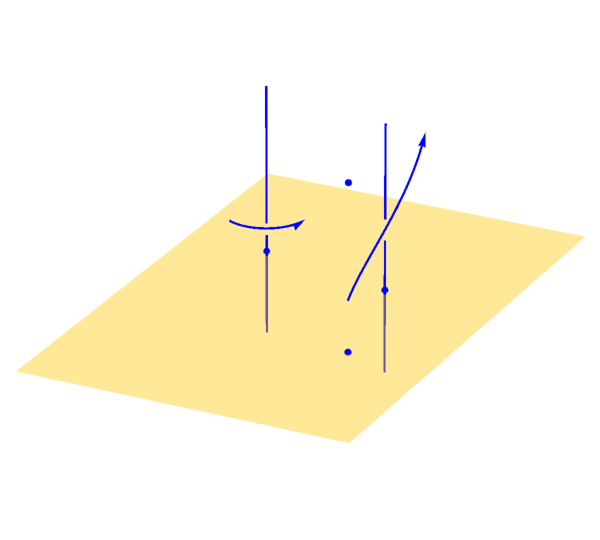}} circle (0pt) ;
\draw[fill] (-1.12,0.633) node[left=-2pt] {\footnotesize$A$} circle (3pt) ;
\draw[fill] (1.62,2.935) node[left=-4pt] {\footnotesize$A'$} circle (3pt) ;
\draw[fill] (1.6,-2.725) node[left=-4pt] {\footnotesize$A''$} circle (3pt) ;
\end{tikzpicture}
\end{minipage}

\item$\ $\\[-35.25pt]
\begin{minipage}{.5\textwidth}
If $\fl$ is a reflection $\sigma^\pi$ in a plane $\pi$
  through $A$, followed by a rotation $\rho^\ell$ about a line $\ell$
  through $A$ perpendicular to $\pi$, then the plane
    $\pi'=\bis\overline{A'A''}$ is fixed by $\fm$. Hence, replacing $A$ by
    a point in $\pi'$, we may assume $A''=A'$.
\end{minipage}
\begin{minipage}{.425\textwidth}
\centering
\begin{tikzpicture}[scale=.225]
\draw[fill] (0,0) node[red] {\includegraphics[width=4.5cm]{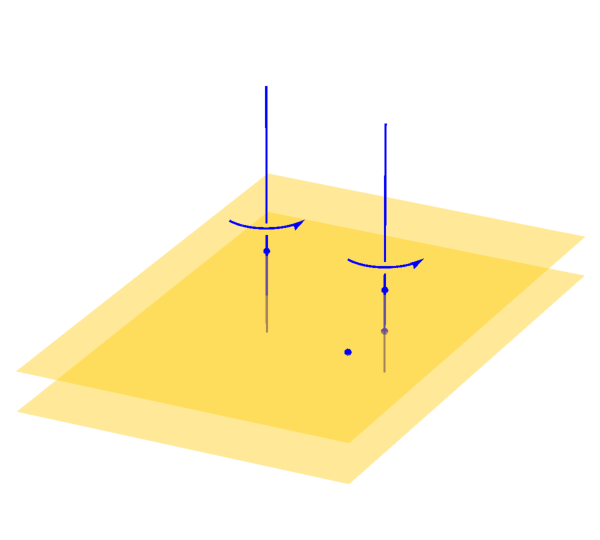}} circle (0pt) ;
\draw[fill] (-1.12,0.633) node[left=-2pt] {\footnotesize$A$} circle (3pt) ;
\draw[fill] (1.6,-2.725) node[left=-4pt] {\footnotesize$A'$} circle (3pt) ;
\draw[fill] (2.83,-2) node[right=-2pt] {\footnotesize$B$} circle (3pt) ;
\end{tikzpicture}
\end{minipage}
\end{enumerate}

Now, If $\rho^\ell$ is not the identity, then
$\tau^{AA'}\circ\rho^{\ell}$ has a fixed point $B\in\pi$ by the Lemma,
and so $\fm=(\tau^{AA'}\circ\rho^{\ell})\circ\sigma^\pi$ also fixes
$B$, contrary to hypothesis. We conclude that
$\fm=\tau^{AA'}\circ\sigma^\pi$, which is a glide plane operation.

Note that in Case~(1) $\fm$ is orientation preserving, while in
Case~(2) it is orientation reversing.
\hfill\qedhere
\end{proof}

\paragraph{\textbf{Acknowledgements.}} 
The authors were partially supported by Centro de Matem\'atica da Universidade do Porto (CMUP), which is financed by national funds through Funda\c{c}\~{a}o para a Ci\^{e}ncia e a Tecnologia (FCT) within project
UIDB/00144/2020.

\end{document}